\documentclass[12pt,oneside]{article}
\usepackage[centertags]{amsmath}
\usepackage{amsfonts}
\usepackage{amssymb}
\usepackage{amsthm}
\usepackage{mathrsfs}
\usepackage{setspace}
\usepackage[papersize={8.5in,11in}]{geometry}
\newtheorem{thm}{Theorem}[section]
\newtheorem{cor}[thm]{Corollary}
\newtheorem{lem}[thm]{Lemma}

\newtheorem{rem}[thm]{Remark}

\newtheorem{defn}[thm]{Definition}

\numberwithin{equation}{section}
\newtheorem{innercustomthm}{Theorem}
\newenvironment{customthm}[1]
{\renewcommand\theinnercustomthm{#1}\innercustomthm}
{\endinnercustomthm}
\numberwithin{equation}{section}
\author{Om  Ahuja, Asena \c{C}etinkaya, and Naveen Kumar Jain}

\title {Analytic functions with conic domains associated with certain generalized $q$-integral operator
\footnotetext{\textit{\textnormal{2010} AMS Mathematics Subject Classification:} 30C45, 30C50, 30C80
\\ \textit{Key words and phrases:} Quantum calculus, $q$-derivative operator, $q$-difference operator, $q$-gamma function, $q$-integral operator, conic domain, $k$-uniformly starlike functions of order gamma, coefficient estimate.
}}
\date{ }

\newcommand{\pe}{{\mathcal P}}
\newcommand{\de}{{\mathbb D}}
\newcommand{\ce}{{\mathbb C}}
\DeclareMathOperator{\RE}{Re}

\begin{document}
\maketitle

\begin{abstract}
In this paper, we define a new subclass of $k$-uniformly starlike functions of order $\gamma,\ (0\leq\gamma<1)$ by using certain generalized	$q$-integral operator. We explore geometric interpretation of the functions in this class by connecting it with conic domains. We also investigate $q$-sufficient coefficient condition, $q$-Fekete-Szeg\"{o} inequalities, $q$-Bieberbach-De Branges type coefficient estimates and radius problem for functions in this class. We conclude this paper by introducing an analogous subclass of  $k$-uniformly convex functions of order $\gamma$ by using the generalized $q$-integral operator. We omit the results for this new class because they can be directly translated from the corresponding results of our main class.
\end{abstract}

\maketitle

\section{Introduction}
Let  $\mathcal{A}$ denote the class of functions of the form
\begin{equation}\label{eq:series}
f(z)=z+\sum_{n=2}^\infty a_{n}z^{n},
\end{equation}
that are analytic in  the open unit disc $\de:=\{z : |z| < 1\}$. Denote by $\pe$ the class of  functions $p$ which are analytic  and have  positive real part in $\mathbb{D}$ with $p(0)=1$. Let  $\Omega$ be the family of Schwarz functions  $w$ which are analytic in  $\de$ satisfying the conditions $w(0)=0$, $|w(z)|<1$ for all $z \in \de$. If $f$ and $g$ are analytic functions in  $\de$, then we say that $f$ is subordinate to $g$, written as $f\prec g$, if there exists a Schwarz function $w\in \Omega\ $ such that $f(z)=g(w(z))$. We also note that if $g$ is univalent in  $\de$, then by subordination rules we get  $f(0) =g(0)$ and $f(\de)\subset g(\de).$

We denote the class $\mathcal{S}$ of all functions in $\mathcal{A}$ that are univalent in  $\de$. Also, let $\mathcal{ST}$ and $\mathcal{CV}$ denote the subclasses of $\mathcal{S}$ that are starlike and convex, respectively. For definitions and properties of these classes, one may refer to the survey article by first author \cite{Ahuja}.

In 1991, Goodman \cite{Good} introduced the concept of uniform  convexity  and uniform starlike functions in  $\mathcal{S}$. In fact, he defined such uniform classes, denoted by $\mathcal{UCV}$ and $\mathcal{UST}$, by their geometrical properties. A function $f$ in $\mathcal{A}$ is said to be uniformly convex (uniformly starlike) in $\de$ if $f$ is in $\mathcal{CV}$ $(\mathcal{ST})$  and has  the property that for every circular arc $\gamma$ contained in $\de$ with center $\xi$ also in $\de$, the arc $f(\gamma)$  is convex (starlike) with respect to $f(\xi)$.

In 1993, Ronning \cite{Ron1} proved the basis for further investigation of the classes $\mathcal{UCV}$ and $\mathcal{UST}$.
\begin{customthm}{A}$(\cite{Ron1})$ If $f\in\mathcal{A}$, then $f\in\mathcal{UCV}$ if and only if
$$\RE\bigg(1+\frac{zf''(z)}{f'(z)}\bigg)> \bigg|\frac{zf''(z)}{f'(z)}\bigg|\ \ (z\in\de).$$
\end{customthm}	
Applying the classic Alexander Theorem found by Alexander \cite{Alex} in 1915, Ronning \cite{Ron1} obtained a characterization of the class $\mathcal{UST}$.
\begin{customthm}{B}$(\cite{Ron1})$ If a function $f$ belongs to $\mathcal{A}$, then $f$ belongs to $\mathcal{UST}$ if and only if
$$\RE\bigg(\frac{zf'(z)}{f(z)}\bigg)> \bigg|\frac{zf'(z)}{f(z)}-1\bigg|\ \ (z\in\de).$$
\end{customthm}
For $k\geq 0$, Kanas \textit{et al.} \cite{Kanas1} introduced the class of  $k$-uniformly starlike functions denoted by $k$-$\mathcal{UST}$. Such a class consists of functions $f\in\mathcal{A}$ that satisfy the inequality
$$\RE\bigg(\frac{zf'(z)}{f(z)}\bigg)> k\bigg|\frac{zf'(z)}{f(z)}-1\bigg|\ \ (z\in\de).$$
We note that $1$-$\mathcal{UST}\equiv\mathcal{UST}$. In 1997, Bharati \textit{et al.} \cite{Bharti} introduced  the class of  $k$-uniformly starlike functions of order $\gamma, (0\leq\gamma<1)$ for functions in the class $k$-$\mathcal{UST}(\gamma)$.
\begin{defn}\label{def1.1}$(\cite{Bharti})$ Let $0\leq\gamma<1$ and $k\geq0$. A function  $f\in\mathcal{A}$ is said to be in  $k$-$\mathcal{UST}(\gamma)$, called $k$-uniformly starlike functions of order $\gamma$, if $f$ satisfies the inequality
$$\RE\bigg(\frac{zf'(z)}{f(z)}\bigg)> k\bigg|\frac{zf'(z)}{f(z)}-1\bigg|+\gamma\ \ (z\in\de).$$	
\end{defn}
We note that 	$k$-$\mathcal{UST}(0)\equiv k$-$\mathcal{UST}$. The class $1$-$\mathcal{UST}(\gamma)$ was investigated  in \cite{Ali} and \cite{Ron2}.

We next  recall some concepts and notations of $q$-calculus that we need to define a new class which connects $k$-$\mathcal{UST}(\gamma)$ and a generalized integral operator defined by $q$-calculus.

Quantum calculus (or $q$-calculus) is a theory of calculus where smoothness is not required. A systematic study of $q$-differentation and $q$-integration was initiated by Jackson  \cite{Jackson1909, Jackson1909F}. The $q$-derivative (or $q$-difference) operator  of a function $f$, defined on a subset of $\ce$, is defined by
\begin{equation*}\label{eq1.0}
	(D_qf)(z)=\begin{cases}
              \frac{f(z)-f(qz)}{(1-q)z}, & \mbox{if } z\neq0 \\
              f'(0), & \mbox{if } z=0,
            \end{cases}
\end{equation*}
where $q\in(0,1)$. Note that $\lim_{q\rightarrow1^- }(D_qf)(z)=f'(z)$ if $f$ is  differentiable at $z$. Under the hypothesis of the definition of $q$-derivative, we have the following rules:
\begin{equation*}
D_{q}(af(z)\pm bg(z))= aD_{q}f(z)\pm bD_{q}g(z)\  (a,b\in\ce),
\end{equation*}
\begin{equation}\label{eq:leibniz}
D_{q}(f(z).g(z))=f(qz)D_{q}g(z)+ g(z)D_{q}f(z).
\end{equation}
It easily follows that if a  function $f$ is  given by (\ref{eq:series}), then
$$(D_qf)(z)=\sum_{n=1}^{\infty}[n]_qa_nz^{n-1},\quad\quad
D_q(zD_qf(z)) =\sum_{n=1}^{\infty}[n]_q^2 a_nz^{n-1},$$
where
\[[n]_q=\frac{1-q^n}{1-q}\]
is called $q$-number or $q$-bracket of $n$. Clearly,  $\lim_{q\rightarrow 1^{-}}[n]_q=n$. In \cite{Jackson1909F}, Jackson defined  $q$-integral of a function $f$  as follows:
$$\int_0^x f(t)d_{q} t=x(1-q) \sum_{{n=0}}^{\infty} q^{n} f(xq^{n}),$$
provided the series converges. It is known that $q$-gamma function is given by
\begin{equation}\label{eq.gamma}
\Gamma_q(x+1)=[x]_q\Gamma_q(x), \quad \Gamma_q(x+1)=[x]_q!,
\end{equation}
where  $q$-factorial $[x]_q!$ is given by
$$
[x]_{q}!= \begin{cases}
[x]_{q}[x-1]_{q}\dots[2]_{q}[1]_q, & \text{if $x\geq 1$}\\
1 , & \text{if $x=0$}.
\end{cases}
$$
Note that in the limiting case when $q\rightarrow 1^-$,  $\Gamma_q(x)\rightarrow \Gamma(x)$; see \cite{Jackson1904}.

The $q$-beta function has the $q$-integral representation, which is a $q$-analogue of Euler's formula (see \cite{Jackson1909F}):
\begin{equation}\label{eq:beta}
B_q(t,s)=\int_0^1x^{t-1}(1-qx)_q^{s-1}d_qx,\ \ (0<q<1; t,s>0).
\end{equation}
Jackson \cite{Jackson1904} also   showed that the $q-$beta function defined by the  formula
\begin{equation}\label{eq:beta2}
B_q(t,s)=\frac{\Gamma_q(t)\Gamma_q(s)}{\Gamma_q(t+s)},
\end{equation}
 tends to $B(t,s)$ as $q\rightarrow1^-.$ Also, the Gauss $q$-binomial coefficients are given by
\begin{equation}\label{eq:binom}
\binom{n}{k}_q:=\frac{[n]_q!}{[k]_q![n-k]!}.
\end{equation}
For more details, one may refer to \cite{Gasper2004, Kac2002}.

In recent years, quantum calculus approach has led to a great development in geometric function theory. Ahuja and \c{C}etinkaya \cite{ahu} recently wrote a survey on the use of quantum calculus approach in mathematical sciences and its role in geometric function theory. One may also refer to the recent paper by Ahuja \textit{et al.} \cite{anand}.

Motivated by Jung \textit{et al.}  \cite{Ki},  Mahmood \textit{et al.} \cite{Mah} introduced the generalized $q$-integral operator $\mathcal{\chi}^{\alpha}_{\beta,q}f:\mathcal{A}\rightarrow\mathcal{A}$ defined by
\begin{equation}\label{eq:intopX}
\mathcal{\chi}^{\alpha}_{\beta,q} f(z)=\binom{\alpha+\beta}{\beta}_q\frac{[\alpha]_q}{z^\beta}\int_0^z \bigg(1-\frac{qt}{z}\bigg)_q^{\alpha-1}t^{\beta-1}f(t)d_qt,
\end{equation}
where $\alpha>0, \beta>-1, q\in(0,1)$. Using (\ref{eq.gamma}), (\ref{eq:beta}), (\ref{eq:beta2}), and (\ref{eq:binom}), they observed that
\begin{equation}\label{eq:powerseries}
\mathcal{\chi}^{\alpha}_{\beta,q} f(z)=z+\sum_{n=2}^\infty\frac{\Gamma_q(\beta+n)}{\Gamma_q(\alpha+\beta+n)}\frac{\Gamma_q(\alpha+\beta+1)}{\Gamma_q(\beta+1)}a_{n}z^{n}.
\end{equation}
For special values of the parameters, the generalized integral operator \eqref{eq:powerseries} gives the  following known integral operators as special cases:
\begin{itemize}
	\item[(i)] For $q\rightarrow 1^-$, the operator $\mathcal{\chi}^{\alpha}_{\beta,q}f$ reduces to  the integral operator $\mathcal{\chi}^{\alpha}_{\beta}f$ defined in \cite{Ki} by
	{\footnotesize \begin{equation}
     \mathcal{\chi}^{\alpha}_{\beta}f(z)=\binom{\alpha+\beta}{\beta}\frac{\alpha}{z^\beta}\int_0^z \bigg(1-\frac{t}{z}\bigg)^{\alpha-1}t^{\beta-1}f(t)dt =z+\sum_{n=2}^\infty\frac{\Gamma(\beta+n)}{\Gamma(\alpha+\beta+n)}\frac{\Gamma(\alpha+\beta+1)}{\Gamma(\beta+1)}a_{n}z^{n}.
	\end{equation}}
	\item[(ii)] For $\alpha=1$, the operator $\mathcal{\chi}^{\alpha}_{\beta,q}f$ yields $q$-Bernardi integral operator $J_{\beta,q}f$ defined in \cite{Noor} by 
	\begin{equation}
	J_{\beta,q}f=\frac{[1+\beta]_q}{z^\beta}\int_0^zt^{\beta-1}f(t)d_qt=\sum_{n=1}^\infty\frac{[1+\beta]_q}{[n+\beta]_q}a_{n}z^{n}.
	\end{equation}
	\item[(iii)] For  $\alpha=1, q\rightarrow 1^-$, the operator $\mathcal{\chi}^{\alpha}_{\beta,q}f$ gives Bernardi integral operator $J_{\beta}f$  defined  in \cite{Ber} by
	\begin{equation}\label{eq 1}
    J_\beta f(z)=\frac{1+\beta}{z^\beta}\int_0^zt^{\beta-1}f(t)dt=\sum_{n=1}^\infty\frac{1+\beta}{n+\beta}a_{n}z^{n}.
    \end{equation}
   \item[(iv)] For  $\alpha=1, \beta=0, q\rightarrow 1^-$, the operator $\mathcal{\chi}^{\alpha}_{\beta,q}f$ reduces to the Alexander integral operator $J_{0}f$ given in \cite{Darus} by 
   \begin{equation}\label{eq 2}
   J_0f(z)=\int_0^z\frac{f(t)}{t}dt=z+\sum_{n=2}^\infty\frac{1}{n}a_{n}z^{n}.
    \end{equation}
\end{itemize}

Making use of  $q$-integral operator,  $\mathcal{\chi}^{\alpha}_{\beta,q}f$ and $k$-uniformly starlike function of order $\gamma$ given in Definition \ref{def1.1}, we define the the new class $k$-$\mathcal{JUST}(q;\alpha,\beta,\gamma)$.
\begin{defn}\label{def1.2} Let \ $0\leq\gamma<1,$ $q\in(0,1)$, $k\geq 0$, $\alpha>0, \beta>-1$. A function $f\in\mathcal{A}$ is in the class $k$-$\mathcal{JUST}(q;\alpha,\beta,\gamma)$  if and only if
	\begin{equation}\label{eq:def1}
	\RE\bigg(\frac{zD_q(\mathcal{\chi}^{\alpha}_{\beta,q}f(z))}{\mathcal{\chi}^{\alpha}_{\beta,q}f(z)}\bigg)> k\bigg|\frac{zD_q(\mathcal{\chi}^{\alpha}_{\beta,q}f(z))}{\mathcal{\chi}^{\alpha}_{\beta,q}f(z)}-1\bigg|+\gamma\ \ (z\in\de)
	\end{equation}
where $\mathcal{\chi}^{\alpha}_{\beta,q}f(z)$ is given by (\ref{eq:powerseries}).
\end{defn}
In what follows, we shall first look at the geometric interpretation of (\ref{eq:coeff}). Since
\begin{equation}\label{eq:coeff}
\begin{split}
\frac{zD_q(\mathcal{\chi}^{\alpha}_{\beta,q}f(z))}{\mathcal{\chi}^{\alpha}_{\beta,q}f(z)}=1+([2]_q-1)\psi_{2}a_{2}z+\big[([3]_q-1)\psi_{3}a_{3}-([2]_q-1)\psi^2_{2}a^2_{2} \big]z^2+\cdots
\end{split}
\end{equation}
where
\[\psi_{n}=\frac{\Gamma_q(\beta+n)}{\Gamma_q(\alpha+\beta+n)}\frac{\Gamma_q(\alpha+\beta+1)}{\Gamma_q(\beta+1)} \quad (n\geq 2).
\]
It follows that for $z=0$, we have $$\frac{zD_q(\mathcal{\chi}^{\alpha}_{\beta,q}f(z))}{\mathcal{\chi}^{\alpha}_{\beta,q}f(z)}=1.$$
Thus $p(z)=zD_q(\mathcal{\chi}^{\alpha}_{\beta,q}f(z))/\mathcal{\chi}^{\alpha}_{\beta,q}f(z)$, where $p$ belongs to the class $\pe$. Therefore (\ref{eq:def1}) is equivalent to
$$\RE\ p(z)>k|p(z)-1|+\gamma  \ (z\in\de).$$
Thus, $p(z)$ takes the values in the conic domain $\Omega_{k,\gamma}$ defined by
$$
\Omega_{k,\gamma}=\bigg\{ u+iv:u>k\sqrt{(u-1)^2+v^2}+\gamma, \  0\leq\gamma<1, k\geq 0  \bigg\}.
$$
Note that  $1\in\Omega_{k,\gamma}$ and $\partial\Omega_{k,\gamma}$ is a curve defined by
$$\partial\Omega_{k,\gamma}=\bigg\{ u+iv: (u-\gamma)^2=k^2(u-1)^2+k^2v^2, \  0\leq\gamma<1, k\geq 0  \bigg\}.$$
Elementary computations show that $\partial\Omega_{k,\gamma}$ represents a conic section symmetric about the real axis. Hence $\Omega_{k,\gamma}$ is an elliptic domain for $k>1$, parabolic domain for $k=1$, hyperbolic domain for $0<k<1$ and a right half plane $u>\gamma$ for $k=0$.

Denote by $\pe(p_{k,\gamma})$ the family of functions $p$ such that $p\in\pe$ and $p\prec p_{k,\gamma}$ in $\de$, where  $p_{k,\gamma}$ maps $\de$ conformally onto the domain $\Omega_{k,\gamma}$.

Kanas \textit{et al.} \cite{Kanas, Kanas1}  found that the  function $p_{k,\gamma}(z)$  plays a role of extremal function of the class $\pe(p_{k,\gamma})$ and is given by
\begin{equation}\label{eq:h(z)}
p_{k,\gamma}(z)=
\begin{cases}
\frac{1+(1-2\gamma)z}{1-z},&\text{$(k=0)$}\\
\frac{1-\gamma}{1-k^2}\cos\bigg\{\frac{2}{\pi}(\cos^{-1}k) i\log\frac{1+\sqrt{z}}{1-\sqrt{z}}\bigg\}-\frac{k^2-\gamma}{1-k^2}, &\text{$(0<k<1)$}\\
1+\frac{2(1-\gamma)}{\pi^2}\bigg(\log\frac{1+\sqrt{z}}{1-\sqrt{z}}\bigg)^2, &\text{$(k=1)$}\\
\frac{1-\gamma}{k^2-1}\sin\bigg\{\frac{\pi}{2K(t)}\int_0^{u(z)/\sqrt{t}}\frac{dx}{\sqrt{1-x^2}\sqrt{1-t^2x^2}}\bigg\}+\frac{k^2-\gamma}{k^2-1}, &\text{$(k>1)$},
\end{cases}
 \\
\end{equation}
where $u(z)=(z-\sqrt{t})/(1-\sqrt{t}z), t\in(0,1)$ and $t$ is chosen such that $k=\cosh\frac{\pi K'(t)}{4K(t)}$ and $K(t)$ is Legendre's complete elliptic integral of the first kind and $K'(t)$ is complementary integral of $K(t)$. Furthermore, since $p_{k,\gamma}(\de)=\Omega_{k,\gamma}$ and $p_{k,\gamma}(\de)$ is convex univalent in $\de$ (see \cite{Kanas}), it follows that (\ref{eq:def1})  is equivalent to
\begin{equation}\label{eq:subor}
\frac{zD_q(\mathcal{\chi}^{\alpha}_{\beta,q}f(z))}{\mathcal{\chi}^{\alpha}_{\beta,q}f(z)}\prec p_{k,\gamma}(z).
\end{equation}

\begin{rem}\label{rem 1} For special values of parameters $q$, $\alpha$, $\beta$, $\gamma$ and $k$, we get the following new classes as  special cases of  Definition \ref{def1.2}; for example:
	\begin{itemize}
		\item[1.] If $q\rightarrow 1^{-}$ and $\alpha=1$, we get $k$-$\mathcal{JUST}(\beta,\gamma):=\lim(k$-$\mathcal{JUST}(q;1,\beta,\gamma))$  with Bernardi operator \eqref{eq 1}
		\item[2.] If $q\rightarrow 1^{-}$, $\alpha=1$, $\beta=0$ and $k=0$, we get  $\mathcal{JUST}(\gamma):=\lim (0$-$\mathcal{JUST}(q;1,0,\gamma))$  with Alexander operator \eqref{eq 2}
	\end{itemize}
\end{rem}

In this paper, we shall investigate the class $k$-$\mathcal{JUST}(q;\alpha,\beta,\gamma)$. In particular, we obtain $q$-sufficient coefficient condition, $q$-Fekete-Szeg\"{o} inequalities,  $q$-Bieberbach-De Branges type coefficient estimates and solve radius problem for the functions in this class. In the concluding section, we introduce another new class $k$-$\mathcal{JUCV}(q;\alpha,\beta,\gamma)$ and omit the results for this class, because analogous results can be directly translated from the coresponding results found in Section 2 for the class $k$-$\mathcal{JUST}(q;\alpha,\beta,\gamma)$.

Unless otherwise stated, we assume in the reminder of the article that  $0\leq\gamma<1,$ $q\in(0,1)$, $k\geq 0$, $\alpha>0, \beta>-1$ and $z\in\de$.
\section{Main Results}
We first obtain $q$-sufficient coefficient condition for the functions belonging to the class $k$-$\mathcal{JUST}(q;\alpha,\beta,\gamma)$.
\begin{thm}\label{thm1} If a function $f$ defined by (\ref{eq:series}) satisfies the inequality
	\begin{equation}\label{eq:2.1}
		\sum_{n=2}^\infty \big([n]_q(k+1)-(k+\gamma)\big)\frac{\Gamma_q(\beta+n)}{\Gamma_q(\alpha+\beta+n)}\frac{\Gamma_q(\alpha+\beta+1)}{\Gamma_q(\beta+1)}|a_{n}|\leq 1-\gamma,
	\end{equation}	
	then $f$ belongs to  $k$-$\mathcal{JUST}(q;\alpha,\beta,\gamma)$. The result is sharp.	
\end{thm}
\begin{proof} To show that  $f\in k$-$\mathcal{JUST}(q;\alpha,\beta,\gamma)$, it suffices to prove that
$$
k\bigg|\frac{zD_q(\mathcal{\chi}^{\alpha}_{\beta,q}f(z))}{\mathcal{\chi}^{\alpha}_{\beta,q}f(z)}-1\bigg|-\RE\bigg(\frac{zD_q(\mathcal{\chi}^{\alpha}_{\beta,q}f(z))}{\mathcal{\chi}^{\alpha}_{\beta,q}f(z)}-1\bigg)\leq 1-\gamma.$$	
We note that
	\begin{align}\label{eq:2.2}
		\bigg|\frac{zD_q(\mathcal{\chi}^{\alpha}_{\beta,q}f(z))}{\mathcal{\chi}^{\alpha}_{\beta,q}f(z)}-1\bigg|&=\bigg|\frac{zD_q(\mathcal{\chi}^{\alpha}_{\beta,q}f(z))-\mathcal{\chi}^{\alpha}_{\beta,q}f(z)}{\mathcal{\chi}^{\alpha}_{\beta,q}f(z)}\bigg|\nonumber\\
		\nonumber
		\end{align}
		\begin{align}	
		&=\bigg|\frac{\sum_{n=2}^\infty([n]_q-1)\frac{\Gamma_q(\beta+n)}{\Gamma_q(\alpha+\beta+n)}\frac{\Gamma_q(\alpha+\beta+1)}{\Gamma_q(\beta+1)}a_{n}z^{n}}{z+\sum_{n=2}^\infty \frac{\Gamma_q(\beta+n)}{\Gamma_q(\alpha+\beta+n)}\frac{\Gamma_q(\alpha+\beta+1)}{\Gamma_q(\beta+1)}a_{n}z^{n}}\bigg|\nonumber\\
		\nonumber\\
		&\leq\frac{\sum_{n=1}^\infty\big([n]_q-1\big)\frac{\Gamma_q(\beta+n)}{\Gamma_q(\alpha+\beta+n)}\frac{\Gamma_q(\alpha+\beta+1)}{\Gamma_q(\beta+1)}|a_{n}|}{1-\sum_{n=1}^\infty\frac{\Gamma_q(\beta+n)}{\Gamma_q(\alpha+\beta+n)}\frac{\Gamma_q(\alpha+\beta+1)}{\Gamma_q(\beta+1)}|a_{n}|}.
	\end{align}
In view of (\ref{eq:2.1}), it follows that	
$$1-\sum_{n=1}^\infty\frac{\Gamma_q(\beta+n)}{\Gamma_q(\alpha+\beta+n)}\frac{\Gamma_q(\alpha+\beta+1)}{\Gamma_q(\beta+1)}|a_{n}|>0.$$	
Using  (\ref{eq:2.2}), we have	
	\begin{align}
		& k\bigg|\frac{zD_q(\mathcal{\chi}^{\alpha}_{\beta,q}f(z))}{\mathcal{\chi}^{\alpha}_{\beta,q}f(z)}-1\bigg|-\RE\bigg(\frac{zD_q(\mathcal{\chi}^{\alpha}_{\beta,q}f(z))}{\mathcal{\chi}^{\alpha}_{\beta,q}f(z)}-1\bigg)\nonumber\\
		& \ \ \leq k\bigg|\frac{zD_q(\mathcal{\chi}^{\alpha}_{\beta,q}f(z))}{\mathcal{\chi}^{\alpha}_{\beta,q}f(z)}-1\bigg|+\bigg|\frac{zD_q(\mathcal{\chi}^{\alpha}_{\beta,q}f(z))}{\mathcal{\chi}^{\alpha}_{\beta,q}f(z)}-1\bigg|\nonumber\\
		& \ \ \leq (k+1)\bigg|\frac{zD_q(\mathcal{\chi}^{\alpha}_{\beta,q}f(z))-\mathcal{\chi}^{\alpha}_{\beta,q}f(z)}{\mathcal{\chi}^{\alpha}_{\beta,q}f(z)}\bigg|\nonumber\\
		&\ \ \leq(k+1)\bigg\{\frac{\sum_{n=2}^\infty\big([n]_q-1\big)\frac{\Gamma_q(\beta+n)}{\Gamma_q(\alpha+\beta+n)}\frac{\Gamma_q(\alpha+\beta+1)}{\Gamma_q(\beta+1)}|a_{n}|}{1-\sum_{n=1}^\infty\frac{\Gamma_q(\beta+n)}{\Gamma_q(\alpha+\beta+n)}\frac{\Gamma_q(\alpha+\beta+1)}{\Gamma_q(\beta+1)}|a_{n}|}\bigg\}\nonumber\\
		&\ \ \leq 1-\gamma\nonumber
	\end{align}	
	which proves (\ref{eq:2.1}).

For sharpness, consider the function $f_n:\mathbb{D}\rightarrow\mathbb{C}$ defined by
\[f_n(z)= z-\frac{(1-\gamma)\Gamma_q(\alpha+\beta+n)}{\big([n]_q(k+1)-(k+\gamma)\big)\Gamma_q(\beta+n)}\frac{\Gamma_q(\beta+1)}{ \Gamma_q(\alpha+\beta+1)}z^n.\]	
Since
\begin{equation*}
\RE\bigg(\frac{zD_q(\mathcal{\chi}^{\alpha}_{\beta,q}f(z))}{\mathcal{\chi}^{\alpha}_{\beta,q}f(z)}\bigg)=
\RE\left(\frac{[n]_q(k+1)-(k+\gamma)-(1-\gamma)[n]_qz^{n-1}}{[n]_q(k+1)-(k+\gamma)-(1-\gamma)z^{n-1}}\right)
> \frac{k+\gamma}{k+1}
\end{equation*}
and
\begin{align*}
k\bigg|\frac{zD_q(\mathcal{\chi}^{\alpha}_{\beta,q}f(z))}{\mathcal{\chi}^{\alpha}_{\beta,q}f(z)}-1\bigg|
&=k\left|\frac{(1-\gamma)(1-[n]_q)z^{n-1}}{[n]_q(k+1)-(k+\gamma)-(1-\gamma)z^{n-1}}\right|\\
&<\frac{k(1-\gamma)}{k+1},
\end{align*}
it follows that $f_n\in k$-$\mathcal{JUST}(q;\alpha,\beta,\gamma)$. Also, it is easy to shows that the  equality holds in \eqref{eq:2.1} for the  function $f_n$. Thus the result is sharp.
\end{proof}
\begin{cor} If $f(z)=z+a_nz^n$ and
\[|a_n|\leq\frac{(1-\gamma)\Gamma_q(\alpha+\beta+n)}{\big([n]_q(k+1)-(k+\gamma)\big)\Gamma_q(\beta+n)}\frac{\Gamma_q(\beta+1)}{ \Gamma_q(\alpha+\beta+1)}, \quad (n\geq2)\] then $f\in k$-$\mathcal{JUST}(q;\alpha,\beta,\gamma)$.
\end{cor}
Using Remark \ref{rem 1}.1 and Remark \ref{rem 1}.2, Theorem \ref{thm1} gives the following new results.
\begin{cor} If a function $f$ defined by (\ref{eq:series}) is in the class $k$-$\mathcal{JUST}(\beta,\gamma)$, then
$$	\sum_{n=2}^\infty \big(n(k+1)-(k+\gamma)\big)\frac{1+\beta}{n+\beta}|a_{n}|\leq 1-\gamma.$$
\end{cor}
\begin{cor} If a function $f$ defined by (\ref{eq:series}) is in the class $\mathcal{JUST}(\gamma)$, then
	$$	\sum_{n=2}^\infty \big(n-\gamma)\frac{1}{n}|a_{n}|\leq 1-\gamma.$$
\end{cor}	
In order to determine $q$-Fekete-Szeg\"{o} inequalities for the  functions in the class $k$-$\mathcal{JUST}(q;\alpha,\beta,\gamma)$, we need next three lemmas.
\begin{lem}\label{lem:pk}$(\cite{Al, Kanas})$ Let $k\geq 0$ be fixed and  $p_{k,\gamma}$ defined by (\ref{eq:h(z)}). If
$$p_{k,\gamma}(z)=1+P_1z+P_2z^2+...,$$
then
	\begin{equation}\label{eq:p1(z)}
	P_{1}(z)=
	\begin{cases}
	\frac{8(1-\gamma)(\arccos k)^2}{\pi^2(1-k^2)},&\text{ $0\leq k<1$}\\
	\frac{8(1-\gamma)}{\pi^2}, &\text{ $k=1$}\\
	\frac{\pi^2(1-\gamma)}{4\sqrt{t}(1+t)K^2(t)(k^2-1)}, &\text{$k>1$},
	\end{cases}
	\\
	\end{equation}	
and
\begin{equation}\label{eq:p2(z)}
P_{2}(z)=
\begin{cases}
\frac{(A^2+2)}{3}P_1,&\text{ $0\leq k<1$}\\
\frac{2}{3}P_1, &\text{ $k=1$}\\
\frac{4K^2(t)(t^2+6t+1)-\pi^2}{24\sqrt{t}K^2(t)(1+t)}P_1, &\text{$k>1$},
\end{cases}
\\
\end{equation}	
where $A=\frac{2}{\pi}(\cos^{-1}k)$ and $t\in(0,1)$ are chosen such that $k=\cosh(\pi K'(t)/4K(t))$ and $K(t)$ is Legendre's complete elliptic integral of the first kind and $K'(t)$ is complementary integral of $K(t)$.
\end{lem}
\begin{lem} \label{lem:fekete}$(\cite [Lemma\ 3, p.254] {LIB1})$
 If $p(z)=1+c_1 z+c_2z^2+c_3z^3+\cdots$ is in class $\mathcal{P}$ and $\eta$ is a complex number, then \[|c_2-\eta c_1^2|\leq 2\max\{1,|2\eta-1|\}.\] The result is sharp for the functions  $p(z)=(1+z^2)/(1-z^2)$ and $p(z)=(1+z)/(1-z).$
\end{lem}
\begin{lem}\label{lemma3}$(\cite[Lemma \ 1,  p.162]{Ma})$
	If $p(z)=1+c_1 z+c_2z^2+c_3z^3+\cdots$ is in class $\mathcal{P}$ and $\eta$ is a real number, then
	\begin{equation*}
		|c_2-\eta c_1^2| \leq\left\{
		\begin{array}{cc}
			-4\eta+2 &\text{if}\quad \eta\leq 0,\\
			2 &\quad\text{if}\quad 0\leq \eta\leq 1,\\
			4\eta-2 & \text{if}\quad \eta\geq 1.
		\end{array}
		\right.\\
	\end{equation*}
	When $\eta<0$ and $\eta>1$, equality holds if and only if $p(z)=(1+z)/(1-z)$ or one of its rotations. If $0<\eta<1$, then equality holds if and only if $p(z)=(1+z^2)/(1-z^2)$ or one of its rotations. If $\eta=0$, equality holds if and only if \[p(z)=\frac{1+\lambda}{2}\left(\frac{1+z}{1-z}\right)+\frac{1-\lambda}{2}\left(\frac{1-z}{1+z}\right),\quad 0\leq \lambda\leq 1\] or one of its rotations. If $\eta=1$, equality holds if and only if $p(z)$ is the reciprocal of one of the functions such that the equality holds in the case $\eta=0$.
\end{lem}
\begin{thm}\label{thm2} Let  $k\geq 0$ and $f\in k$-$\mathcal{JUST}(q;\alpha,\beta,\gamma)$, where $f$ is of the form (\ref{eq:series}). Then, for a complex number $\eta$, $q$-Fekete-Szeg\"{o} inequality is given by
$$	
|a_{3}-\eta a_{2}^2|\leq\frac{P_1}{2([3]_q-1)\psi_{3}}\max\{1,|2\nu-1|\},
$$
where
\begin{equation}\label{eq:nu}
\nu=\frac{1}{2}-\frac{P_2}{2P_1}-\frac{P_1}{2([2]_q-1)}+\eta\frac{P_1([3]_q-1)\psi_{3}}{2([2]_q-1)^2\psi^2_{2}},
\end{equation}
\[\psi_{n}=\frac{\Gamma_q(\beta+n)}{\Gamma_q(\alpha+\beta+n)}\frac{\Gamma_q(\alpha+\beta+1)}{\Gamma_q(\beta+1)}, \quad (n=2,3),
\]
 and where $P_1$ and $P_2$ are given by Lemma \ref{lem:pk}.
\end{thm}
\begin{proof} If $f\in k$-$\mathcal{JUST}(q;\alpha,\beta,\gamma)$, then there is a Schwarz function $w$, analytic in $\de$ with $w(0)=0$ and $|w(z)|< 1$  such that
$$\frac{zD_q(\mathcal{\chi}^{\alpha}_{\beta,q}f(z))}{\mathcal{\chi}^{\alpha}_{\beta,q}f(z)}= p_{k,\gamma}(w(z)).$$		
Define the function $p$ by
$$p(z)=\frac{1+w(z)}{1-w(z)}=1+c_1z+c_2z^2+...\quad(z\in\de).$$
Since  $p\in\pe$ is a function with $p(0)=1$ and $\RE(p(z))>0$, we get
\begin{align}\label{eq:h,k,p}
	p_{k,\gamma}(w(z))& =p_{k,\gamma}\bigg(\frac{p(z)-1}{p(z)+1}\bigg) \nonumber \\
	&=p_{k,\gamma}\bigg(\frac{c_1}{2}z+\frac{1}{2}\bigg(c_2-\frac{c_1^2}{2}\bigg)z^2+\frac{1}{2}\bigg(c_3-c_1c_2+\frac{c_1^3}{4}\bigg)z^3+...\bigg) \nonumber \\
	&=1+\frac{P_1c_1}{2}z+\bigg(\frac{P_1}{2}\bigg(c_2-\frac{c_1^2}{2}\bigg)+\frac{P_2c_1^2}{4}\bigg)z^2+...
\end{align}
Comparing the coefficients of (\ref{eq:h,k,p}) and (\ref{eq:coeff}), we get
$$a_{2}=\frac{P_1c_1}{2([2]_q-1)\psi_{2}}$$
and
$$a_{3}=\frac{P_1}{2([3]_q-1)\psi_{3}}\bigg\{c_2-\frac{c^2_1}{2}+\frac{c^2_1}{2}\bigg(\frac{P_2}{P_1}+\frac{P_1}{[2]_q-1}\bigg)\bigg\}.$$
For any complex number $\eta$, we have
{\small
\begin{align}\label{eq:fekete}
a_{3}-\eta a_{2}^2&=\frac{P_1}{2([3]_q-1)\psi_{3}}\bigg\{c_2-\frac{c^2_1}{2}+\frac{c^2_1}{2}\bigg(\frac{P_2}{P_1}+\frac{P_1}{[2]_q-1}\bigg)\bigg\}- \eta\frac{P_1^2c_1^2}{4([2]_q-1)^2\psi^2_{2}}.
\end{align}}
Equation (\ref{eq:fekete}) can be written as:
\begin{equation}\label{eq:a3nua2}
a_{3}-\eta a_{2}^2=\frac{P_1}{2([3]_q-1)\psi_{3}}\{c_2-\nu c_1^2\},
\end{equation}
where $\nu$ is defined by (\ref{eq:nu}). Applying  Lemma \ref{lem:fekete},  proof is completed. The result is sharp for a function $f$ given by
$$\frac{zD_q(\mathcal{\chi}^{\alpha}_{\beta,q}f(z))}{\mathcal{\chi}^{\alpha}_{\beta,q}f(z)}=p_{k,\gamma}(z)\quad \text{or}\quad \frac{zD_q(\mathcal{\chi}^{\alpha}_{\beta,q}f(z))}{\mathcal{\chi}^{\alpha}_{\beta,q}f(z)}=p_{k,\gamma}(z^2). \qedhere$$
\end{proof}
 In view of Remark \ref{rem 1}.1 and Remark \ref{rem 1}.2, we get the following new results as special cases of Theorem \ref{thm2}.
\begin{cor} If a function $f$ defined by (\ref{eq:series}) is in the class $k$-$\mathcal{JUST}(\beta,\gamma)$, then
$$|a_{3}-\eta a_{2}^2|\leq\frac{P_1(3+\beta)}{4(1+\beta)}\max\bigg\{1,\bigg|-\frac{P_2}{P_1}-P_1+\eta\frac{2P_1(2+\beta)^2}{(3+\beta)(1+\beta)}\bigg|\bigg\}.$$	\end{cor}

\begin{cor} If a function $f$ defined by (\ref{eq:series}) is in the class $\mathcal{JUST}(\gamma)$, then
		$$|a_{3}-\eta a_{2}^2|\leq\frac{3P_1}{4}\max\bigg\{1,\bigg|-\frac{P_2}{P_1}-P_1+\eta\frac{8P_1}{3}\bigg|\bigg\}.$$
\end{cor}
\begin{thm} Let  $k\geq 0$ and $f\in k$-$\mathcal{JUST}(q;\alpha,\beta,\gamma)$ where $f$ is of the form (\ref{eq:series}). Then, for a real number $\eta$, we have
\begin{equation*}	
|a_{3}-\eta a_{2}^2|\leq\frac{1}{([3]_q-1)\psi_{3}}\times\begin{cases}
P_2+\frac{P_1^2}{[2]_q-1}-\eta\frac{P_1^2([3]_q-1)\psi_{3}}{([2]_q-1)^2\psi_{2}^2},&if\ \text{$\eta\leq\sigma_1$}\\
P_1, & if\ \text{$\sigma_1\leq \eta\leq \sigma_2$}\\
-P_2-\frac{P_1^2}{[2]_q-1}+\eta\frac{P_1^2([3]_q-1)\psi_{3}}{([2]_q-1)^2\psi_{2}^2}, & if\ \text{$\eta\geq\sigma_2$}\\
\end{cases},
\end{equation*}
where

\[\psi_{n}=\frac{\Gamma_q(\beta+n)}{\Gamma_q(\alpha+\beta+n)}\frac{\Gamma_q(\alpha+\beta+1)}{\Gamma_q(\beta+1)}, \quad (n=2,3),
\]

$$\sigma_1=\frac{([2]_q-1)^2\psi_{2}^2}{([3]_q-1)\psi_{3}}\bigg(\frac{P_2}{P_1^2}+\frac{1}{[2]_q-1}-\frac{1}{P_1}\bigg),$$
$$\sigma_2=\frac{([2]_q-1)^2\psi_{2}^2}{([3]_q-1)\psi_{3}}\bigg(\frac{P_2}{P_1^2}+\frac{1}{[2]_q-1}+\frac{1}{P_1}\bigg),$$
and $P_1$ and $P_2$ are given by Lemma \ref{lem:pk}.
\end{thm}
\begin{proof} Using  (\ref{eq:nu}), (\ref{eq:a3nua2}) and Lemma \ref{lemma3}, we get the proof.  The bounds are sharp as can be seen by defining the following functions for $n\geq2$ and $0\leq \lambda\leq 1$.
	\[\frac{zD_q(\mathcal{\chi}^{\alpha}_{\beta,q}F_n(z))}{\mathcal{\chi}^{\alpha}_{\beta,q}F_n(z)}=p_{k,\gamma}(z^{n-1}),\quad \mathcal{F}_n(0)=\mathcal{F'}_n(0)-1=0,\]
	\[\frac{zD_q(\mathcal{\chi}^{\alpha}_{\beta,q}G_\lambda(z))}{\mathcal{\chi}^{\alpha}_{\beta,q}G_\lambda(z)}=p_{k,\gamma}\left(\frac{z(z+\lambda)}{1+\lambda z}\right), \quad \mathcal{G}_\lambda(0)=\mathcal{G'}_\lambda(0)-1=0,\]
	
	\[\frac{zD_q(\mathcal{\chi}^{\alpha}_{\beta,q}H_\lambda(z))}{\mathcal{\chi}^{\alpha}_{\beta,q}H_\lambda(z)}=p_{k,\gamma}\left(-\frac{z(z+\lambda)}{1+ \lambda z}\right), \quad \mathcal{H}_\lambda(0)=\mathcal{H'}_\lambda(0)-1=0.\]
When $\eta<\psi_1$ or $\eta>\psi_2$, equality holds if and only if $f$ is $\mathcal{F}_2$ or one of its rotations. When $\psi_1<\eta<\psi_2$, equality holds if and only if $f$ is $\mathcal{F}_3$ or one of its rotations. If $\eta=\psi_1$, equality holds if and only if $f$ is $\mathcal{G}_\lambda$ or one of its rotations and if $\eta=\psi_2$, equality holds if and only if $f$ is $\mathcal{H}_\lambda$ or one of its rotations.
\end{proof}

For investigating $q$-Bieberbach-De Branges inequalities, we need the following result called Rogogonki's Theorem.
\begin{lem}$(\cite[Theorem\ 2.3, p.70]{Rog})$\label{thm 1} Let $p(z)=1+\sum_{n=1}^{\infty}c_nz^n$  be subordinate to $p_{k,\gamma}(z)=1+\sum_{n=1}^{\infty}P_nz^n$ in $\de$. If $p_{k,\gamma}$ is univalent  in $\de$ and $p_{k,\gamma}(\de)$ is convex, then
$$|c_n|\leq P_1, \ \ (n\geq1).$$
\end{lem}

\begin{thm} If a function $f$  of the form \eqref{eq:series} belongs to the class $k$-$\mathcal{JUST}(q;\alpha,\beta,\gamma)$, then
\[|a_2|\leq\frac{P_1}{q\psi_2} \quad \text{and} \quad |a_3|\leq\frac{qP_2+P_1^2}{q^2(1+q)\psi_3}.\]
These results are sharp for the function given by (\ref{eq:biebe-bran}).
\end{thm}
\begin{proof}
Let $p(z)=\frac{zD_q(\mathcal{\chi}^{\alpha}_{\beta,q}f(z))}{\mathcal{\chi}^{\alpha}_{\beta,q}f(z)}.$  Using the relation (\ref{eq:powerseries}) for  $p(z)=1+c_1z+c_2z^2+\cdots$, we have
$$
([n]_q-1)\psi_na_n=\sum_{k=1}^{n-1}\psi_ka_kc_{n-k}, \ \ a_1=1.
$$
Comparing the coefficients  for $n=2$ and $n=3$, we get
\begin{equation}\label{eq 13}
a_2=\frac{c_1}{([2]_q-1)\psi_2}\ \text{and}\ a_3=\frac{c_2+c_1\psi_2a_2}{([3]_q-1)\psi_3}.
\end{equation}
It is obvious that
\begin{equation}\label{eq:a2+}
|a_2|=\frac{|c_1|}{([2]_q-1)\psi_2}\leq\frac{P_1}{q\psi_2},
\end{equation}
where $|c_1|\leq P_1$.

Now, Lemma \ref{thm 1}, \eqref{eq 13} together with inequality  $|c_1^2|+|c_2|\leq|P_1|^2+|P_2|$ (see \cite{kan}) yield
\begin{align*}
|a_3|&=\left|\frac{qc_2+c_1^2}{q^2(1+q)\psi_3}\right|\leq\frac{q(|c_2|+|c_1|^2)+(1-q)|c_1^2|}{q^2(1+q)\psi_3}\\
&\leq\frac{q(|P_2|+|P_1|^2)+(1-q)|P_1^2|}{q^2(1+q)\psi_3}\leq\frac{qP_2+P_1^2}{q^2(1+q)\psi_3}.
\end{align*}
\end{proof}
In our next result, we state and prove a $q$-Bieberbach-De Branges inequality.
\begin{thm}\label{thm3} If $f\in k$-$\mathcal{JUST}(q;\alpha,\beta,\gamma)$, then
\begin{equation}\label{eq:bieberbach}
|a_n|\leq\frac{P_1}{([n]_q-1)\psi_n}\prod_{k=1}^{n-2}\bigg(1+\frac{P_1}{([k+1]_q-1)}\bigg),\ \ (n\geq3)
\end{equation}	
where $P_1$ is given by (\ref{eq:p1(z)}) and

\begin{equation}\label{eq1}
\psi_{n}=\frac{\Gamma_q(\beta+n)}{\Gamma_q(\alpha+\beta+n)}\frac{\Gamma_q(\alpha+\beta+1)}{\Gamma_q(\beta+1)}, \quad (n\geq3).
\end{equation}
\end{thm}
\begin{proof} In view of Definition \ref{def1.2}, we can write
	$$\frac{zD_q(\mathcal{\chi}^{\alpha}_{\beta,q}f(z))}{\mathcal{\chi}^{\alpha}_{\beta,q}f(z)}=p(z)\prec p_{k,\gamma}$$
	where $p\in \pe$ is analytic in $\de$. Since   $p(z)=1+\sum_{n=1}^{\infty}c_nz^n$ and $\mathcal{\chi}^{\alpha}_{\beta,q}f$ given by (\ref{eq:powerseries}), we have
	$$z+\sum_{n=2}^{\infty}[n]_q\psi_na_nz^n=\bigg(z+\sum_{n=2}^{\infty}\psi_na_nz^n\bigg)\bigg(1+\sum_{n=1}^{\infty}c_nz^n\bigg),$$
	where $\psi_{n}$ are given by \eqref{eq1}.
	
	Comparing the coefficients of $z^n$ on both sides, we observe	
	$$[n]_q\psi_na_n=a_n+c_1\psi_{n-1}a_{n-1}+c_2\psi_{n-2}a_{n-2}+...+c_{n-2}\psi_2a_2+c_{n-1}$$
	for all integer $n\geq 3.$ Taking absolute value on both sides and applying   Lemma \ref{thm 1},  we have
	\begin{equation*}\label{eq:an}
		|a_n|\leq \frac{P_1}{([n]_q-1)\psi_n}\{1+\psi_2|a_2|+...+\psi_{n-2}|a_{n-2}|+\psi_{n-1}|a_{n-1}|\}.
	\end{equation*}
We will prove \eqref{eq:bieberbach} by using mathematical induction. For $n=2$, the result follows by (\ref{eq:a2+}). Let us assume that (\ref{eq:bieberbach}) is true for $n\leq m$, that is
\begin{align}\label{eq:ant}
|a_m|&\leq \frac{P_1}{([m]_q-1)\psi_m}\{1+\psi_2|a_2|+...+\psi_{m-1}|a_{m-1}|\}\nonumber\\
&\leq\frac{P_1}{([m]_q-1)\psi_m}\prod_{k=1}^{m-2}\bigg(1+\frac{P_1}{([k+1]_q-1)}\bigg).\nonumber
\end{align}
Consider
\begin{align}
|a_{m+1}|&\leq \frac{P_1}{([m+1]_q-1)\psi_{m+1}}\{1+\psi_2|a_2|+...+\psi_{m}|a_{m}|\}\nonumber\\
&\leq \frac{P_1}{([m+1]_q-1)\psi_{m+1}}\bigg\{1+\frac{P_1}{([2]_q-1)}+\frac{P_1}{([3]_q-1)}\bigg(1+\frac{P_1}{([2]_q-1)}\bigg)+...+\nonumber\\
&\quad\quad\quad\frac{P_1}{([m]_q-1)}\prod_{k=1}^{m-2}\bigg(1+\frac{P_1}{([k+1]_q-1)}\bigg)\bigg\}\nonumber\\
&= \frac{P_1}{([m+1]_q-1)\psi_{m+1}}\prod_{k=1}^{m-1}\bigg(1+\frac{P_1}{([k+1]_q-1)}\bigg).\nonumber
\end{align}
Thus \eqref{eq:bieberbach} is true for $n=m+1$. Consequently, mathematical induction shows that \eqref{eq:bieberbach} holds for $n,\ n\geq2$. This completes the proof.
The result is sharp for a function $f$ given by
\begin{equation}\label{eq:biebe-bran}
\frac{zD_q(\mathcal{\chi}^{\alpha}_{\beta,q}f(z))}{\mathcal{\chi}^{\alpha}_{\beta,q}f(z)}=p_{k,\gamma}(z).
\end{equation}
\end{proof}
For different values of the parameters, Theorem \ref{thm3} gives several new results. In particular in view of Remark \ref{rem 1}.1 and Remark \ref{rem 1}.2. Theorem \ref{thm3} gives the following results.
\begin{cor} If a function $f$ defined by (\ref{eq:series}) is in the class $k$-$\mathcal{JUST}(\beta,\gamma)$, then
$$|a_n|\leq\frac{P_1(\beta+n)}{(n-1)(\beta+1)}\prod_{k=1}^{n-2}\bigg(1+\frac{P_1}{k}\bigg),\ \ (n\geq3)$$
\end{cor}
\begin{cor} If a function $f$ defined by (\ref{eq:series}) is in the class $\mathcal{JUST}(\gamma)$, then
	$$|a_n|\leq\frac{nP_1}{(n-1)}\prod_{k=1}^{n-2}\bigg(1+\frac{P_1}{k}\bigg),\ \ (n\geq3)$$
\end{cor}
We now conclude this section by exploring $q$-radius for the functions in the class $k$-$\mathcal{JUST}(q;\alpha,\beta,\gamma)$.
\begin{thm} If $f\in k$-$\mathcal{JUST}(q;\alpha,\beta,\gamma)$, then $f(\de)$ contains an open disc of radius
\begin{equation*}
r=\frac{q\psi_2}{2q\psi_2+P_1},
\end{equation*}
where $P_1$ is given by (\ref{eq:p1(z)}).	
\end{thm}
\begin{proof} Let $w_0\neq 0$ be a complex number such that $f(z)\neq w_0$ for $z\in\de$. Then
$$f_1(z)=\frac{w_0f(z)}{w_0-f(z)}=z+\bigg(a_2+\frac{1}{w_0}\bigg)z^2+...\ .$$
Since $f_1$ is univalent in $\de$, it follows that
$$\bigg|a_2+\frac{1}{w_0}\bigg|\leq2.$$	
By using (\ref{eq:a2+}), we get
$$\bigg|\frac{1}{w_0}\bigg|\leq2+\frac{P_1}{q\psi_2}=\frac{2q\psi_2+P_1}{q\psi_2}.$$	
Consequently, we obtain 	
$$|w_0|\geq\frac{q\psi_2}{2q\psi_2+P_1}.$$	
\end{proof}

\section{Concluding Remarks}
Using the well-known formula (\ref{eq:leibniz})  and replacing $\mathcal{\chi}^{\alpha}_{\beta,q}f(z)$ in (\ref{eq:def1}) by $zD_q(\mathcal{\chi}^{\alpha}_{\beta,q}f(z))$, we obtain a new subclass $k$-$\mathcal{JUCV}(q;\alpha,\beta,\gamma)$ of $k$-uniformly convex functions of order $\gamma$ associated with the generalized $q$-integral operator given by (\ref{eq:intopX}).
\begin{defn} Let $0\leq\gamma<1$, $q\in(0,1)$, $k\geq 0$, $\alpha>0, \beta>-1$. A function $f\in\mathcal{A}$ is in the class $k$-$\mathcal{JUCV}(q;\alpha,\beta,\gamma)$ if and only if
	\begin{equation*}
		\RE\bigg(1+q\frac{zD^2_q(\mathcal{\chi}^{\alpha}_{\beta,q}f(z))}{D_q(\mathcal{\chi}^{\alpha}_{\beta,q}f(z))}\bigg)> k\bigg|q\frac{zD^2_q(\mathcal{\chi}^{\alpha}_{\beta,q}f(z))}{D_q(\mathcal{\chi}^{\alpha}_{\beta,q}f(z))}\bigg|+\gamma \ \ (z\in\de)
	\end{equation*}
where $\mathcal{\chi}^{\alpha}_{\beta,q}f(z)$ is given by (\ref{eq:intopX}) and (\ref{eq:powerseries}).
\end{defn}
Alexander-type relationship between functions of these classes is
$$\mathcal{\chi}^{\alpha}_{\beta,q}f(z)\in k\text{-}\mathcal{JUCV}(q;\alpha,\beta,\gamma)\Leftrightarrow zD_q(\mathcal{\chi}^{\alpha}_{\beta,q}f(z))\in k\text{-}\mathcal{JUST}(q;\alpha,\beta,\gamma);$$
that is,
$$\mathcal{\chi}^{\alpha}_{\beta,q}f(z)\in k\text{-}\mathcal{JUST}(q;\alpha,\beta,\gamma)\Leftrightarrow \int_0^z\frac{\mathcal{\chi}^{\alpha}_{\beta,q}f(z)}{t}dt\in  k\text{-}\mathcal{JUCV}(q;\alpha,\beta,\gamma).$$
In view of the classical Alexander Theorem and the results for the class $k$-$\mathcal{JUST}(q;\alpha,\beta,\gamma)$, it is easy to obtain the corresponding properties for the class $k$-$\mathcal{JUCV}(q;\alpha,\beta,\gamma)$. Therefore, we omit the statements and proofs of the corresponding results of the class $k$-$\mathcal{JUCV}(q;\alpha,\beta,\gamma)$.

\footnotesize
\vspace{2cm}

\textsc{Om P. Ahuja}\\
Department of Mathematical Sciences,\\
Kent State University, Ohio, 44021, U.S.A   \\
e-mail: oahuja@kent.edu  \\

\textsc{Asena \c{C}etinkaya}\\
Department of Mathematics and Computer Sciences,\\
\.{I}stanbul K\"{u}lt\"{u}r University, \.{I}stanbul, Turkey\\
e-mail: asnfigen@hotmail.com\\

\textsc{Naveen Kumar Jain}\\
Department of Mathematics\\
Aryabhatta College,  Delhi, 110021, India\\
e-mail: naveenjain@aryabhattacollege.ac.in


\begin{thebibliography}{1}
\bibitem{Alex} J. W.  Alexander, Functions which map the interior of the unit circle upon simple regions, \textit{Ann. Math.} {\bf 17} (1915/1916), 12-22.		
\bibitem{Ahuja} O. P. Ahuja, The Bieberbach conjecture and its impact on the developments in Geometric  Function Theory, Math. Chronicle, {\bf 15} (1986), 1-28.
\bibitem{ahu}  O. P. Ahuja and   A.  \c{C}etinkaya, Use of Quantum Calculus approach in Mathematical Sciences and its role in geometric function theory. \emph{AIP Conf. Proc.} {\bf 2019}, \emph{2095}, 020001-1--020001-14.
\bibitem{anand} O. P. Ahuja, S. Anand and N. K. Jain, Bohr Radius Problems for Some Classes of Analytic Functions Using Quantum Calculus Approach, Mathematics, {\bf 8}, (2020), 623.
\bibitem{Al} F. M. Al-Oboudi and K. A. Al-Amoudi, On classes of analytic functions related to conic domains, J. Math. Anal. Appl. {\bf 339} (2008), no.~1, 655--667.
\bibitem{Ali}  R. M. Ali and  V. Singh,  Coefficient of parabolic starlike functions of order alpha, In: Comput.Methods Funct. Theory Ser. Approx. Decompos., 1994 (Penang), Vol. 5, World Scientific Publishing, Singapore, (1995), 23--26.	
\bibitem{Ber}  S. D. Bernardi, Convex and starlike univalent functions, Trans. Amer. Math. Soc. {\bf 135} (1969), 429--446.
\bibitem{Bharti} R. Bharati, R. Parvatham\ and\ A. Swaminathan, On subclasses of uniformly convex functions and corresponding class of starlike functions, Tamkang J. Math. {\bf 28} (1997), no.~1, 17--32.
\bibitem{Good} A. W. Goodman, On uniformly convex functions, Ann. Polon. Math. {\bf 56} (1991), no.~1, 87--92.
\bibitem{Gasper2004} G. Gasper\ and\ M. Rahman, {\it Basic hypergeometric series}, second edition, Encyclopedia of Mathematics and its Applications, 96, Cambridge University Press, Cambridge, 2004.
\bibitem{Jackson1904} F. H.  Jackson,  \textit{A generalization of the functions $\Gamma(n)$ and $x^n$}, Proc. Royal Soc. London, {\bf 74} (1904), 64-72.
\bibitem{Jackson1909} F. H. Jackson, {On $q$-functions and a certain difference operator}, Trans. Royal Soc. Edinburgh, {\bf 46} (1908), no. 2, 253-281.
\bibitem{Jackson1909F} F. H. Jackson, \textit{On $q$-definite integrals}, Quart. J. Pure Appl. Math. {\bf 41} (1910), 193-203.
\bibitem{Ki} I. B. Jung, Y. C. Kim\ and\ H. M. Srivastava, The Hardy space of analytic functions associated with certain one-parameter families of integral operators, J. Math. Anal. Appl. {\bf 176} (1993), no.~1, 138--147.
\bibitem{Kac2002} V. Kac\ and\ P. Cheung, {\it Quantum calculus}, Universitext, Springer-Verlag, New York, 2002.
\bibitem{kan}S. Kanas\ and\ A. Wi\'{s}niowska, Conic regions and $k$-uniform convexity. II, Zeszyty Nauk. Politech. Rzeszowskiej Mat. No. 22 (1998), 65--78.
\bibitem{Kanas} S. Kanas and A. Wi\'{s}niowska, {Conic regions and $k$-uniform convexity}, J. Comput. Appl. Math. {\bf 105} (1999), 327-336.
\bibitem{Kanas1} S. Kanas\ and\ A. Wi\'{s}niowska, Conic domains and starlike functions, Rev. Roumaine Math. Pures Appl. {\bf 45} (2000), no.~4, 647--657.
\bibitem{LIB1}R. J. Libera\ and\ E. J. Z\l otkiewicz, Coefficient bounds for the inverse of a function with derivative in $\mathcal {P}$, Proc. Amer. Math. Soc. {\bf 87} (1983), no.~2, 251--257.
\bibitem{Ma}  W. C. Ma\ and\ D. Minda, A unified treatment of some special classes of univalent functions, in {\it Proceedings of the Conference on Complex Analysis (Tianjin, 1992)}, 157--169, Conf. Proc. Lecture Notes Anal., I, Int. Press, Cambridge, MA. 1994.
\bibitem{Mah} S. Mahmood, N. Raza, E. S. A. Abujarad, G. Srivastava, H. M. Srivastava, S. N. Malik, {Geometric properties of certain classes of analytic functions associated with $q$-integral operators}, Symmetry {\bf 11} (2019), 719.
\bibitem{Noor} K. I. Noor, S. Riaz\ and\ M. A. Noor, On $q$-Bernardi integral operator, TWMS J. Pure Appl. Math. {\bf 8} (2017), no.~1, 3--11.
\bibitem{Rog} W. Rogosinski, On the coefficients of subordinate functions, Proc. London Math. Soc. (2) {\bf 48} (1943), 48--82.
\bibitem{Ron1} F. R\o nning, Uniformly convex functions and a corresponding class of starlike functions, Proc. Amer. Math. Soc. {\bf 118} (1993), no.~1, 189--196.
\bibitem{Ron2} F. R\o nning, On starlike functions associated with parabolic regions, Ann. Univ.Mariae Curie-Sklodowska Sect. A 45 (1991), 117–122.
\bibitem{Darus} Z. Shareef, S. Hussain and M. Darus, Convolution operators in the geometric function theory, J. Inequal. Appl. 2012, 2012:213.
\end{thebibliography}
\end{document}